\documentclass[11pt,british,reqno]{amsart}

\usepackage[T1]{fontenc}

\usepackage{babel,eucal,url,amssymb,stmaryrd,booktabs,enumerate,amscd,paralist,cite,color}
\usepackage[pagebackref,colorlinks=true]{hyperref}

\newcommand{\Lie}[1]{\operatorname{\textsl{#1}}}
\newcommand{\lie}[1]{\operatorname{\mathfrak{#1}}}

\markleft{\sc Francisco Martín Cabrera}
\theoremstyle{plain}
\newtheorem{lemma}{Lemma}[section]

\newtheorem{corollary}[lemma]{Corollary}
\newtheorem{theorem}[lemma]{Theorem}
\newtheorem{remark}[lemma]{Remark}

\newcommand{\Gtwo}{\ifmmode{{\rm G}_2}\else{${\rm G}_2$}\fi}

 \newcommand{\cyclic}{\mathop{\kern0.9ex{{+}\kern-2.2ex\raise-.28ex\hbox{\Large\hbox
 {$\circlearrowright$}}}}}

\DeclareMathOperator{\tr}{tr}

\newcommand{\lto}{\longrightarrow}
\newcommand{\sumz}{\sum_{i\in \mathbb Z_7}} 

\def\sideremark#1{\ifvmode\leavevmode\fi\vadjust{\vbox to0pt{\vss
 \hbox to 0pt{\hskip\hsize\hskip1em
 \vbox{\hsize2.5cm\tiny\raggedright\pretolerance10000
 \noindent #1\hfill}\hss}\vbox to8pt{\vfil}\vss}}}%

\newfont{\eusm}{eusm10 scaled \magstep1}
\newfont{\eusmiii}{eusm10 scaled \magstep3}

\newcommand{\w}{{\mathchoice{\,{\scriptstyle\wedge}\,}{{\scriptstyle\wedge}}
      {{\scriptscriptstyle\wedge}}{{\scriptscriptstyle\wedge}}}}

\title{Remarks on some integral formulas for $\Lie{G}_2$-structures}

\author{Francisco~Mart\'\i n~Cabrera}
\address[Francisco~Mart\'\i n~Cabrera]{Departamento de Matemáticas, Estadística e Investigación   Operativa \\
  Universidad de La Laguna\\ 38200 La Laguna, Tenerife, Spain}
\email{fjmartincabrera@gmail.com}

\begin{document}

\maketitle

\markboth{\protect\small \sc  francisco martín cabrera}
         {\protect\small \sc Remarks on some  integral formulas for $\Lie{G}_2$-structures}

\begin{abstract}{\indent}
For seven-dimensional Riemannian manifolds equipped with a $\Lie G_2$-structure, we show in a full detailed way that all  integral formulas and  divergence equations, given by  diverse  authors,  are agree with the ones displayed here in terms of the intrinsic torsion of the $\Lie G_2$-structure. Likewise the components of such an intrinsic torsion  is expressed by means of exterior algebra.

\vspace{3mm}

  \vspace{3mm}

\end{abstract}

\tableofcontents

\section{Introduction}{\indent} 

By using different tools, diverse authors, Bryant in \cite{RB}, Bor and Hernández Lamoneda in  \cite{BHLS1}, Friedrich and Ivanov in \cite{FI2} and Niedziałomski in \cite{KN2},  have deduced integral formulas and divergence equations for Riemannian manifolds equipped with a $\Lie G_2$-structure. They express the scalar curvature of the manifold in terms   related with components  of the intrinsic torsion of the $\Lie G_2$-structure. Such terms are different in each formula. Thus it is natural to ask for the way  in which all mentioned formulas are agree and the correlation  between them.

Our purpose here is to show in a  detailed way that all integral formulas, given by  the above  mentioned authors,  are agree with 
\begin{equation*}
 \int_M s \, \mathrm{vol}_M =  \int_M  9 \|\xi_{(1)} \|^2  - \frac32 \|\xi_{(2)} \|^2  - \frac32 \|\xi_{(3)} \|^2+  \frac{15}2 \|\xi_{(4)} \|^2 \; \mathrm{vol}_M,
\end{equation*} 
where $s$ is the scalar curvature,  $M$ is  a closed  manifold equipped with a $\Lie G_2$-structure and $\xi$ denotes the intrinsic torsion. The terms $\xi_{(i)}$ denote  $\Lie G_2$-components of $\xi$. Likewise, we also prove that all divergence equations, given by such authors, are agree with   
$$
   s = 9 \|\xi_{(1)} \|^2  - \frac32 \|\xi_{(2)} \|^2  - \frac32 \|\xi_{(3)} \|^2+  \frac{15}2 \|\xi_{(4)} \|^2 
   - 6\,  \mathrm{div} \, \sumz \xi_{e_i} e_i.
$$

Correlations between the   integral formulas and  between the divergence equations have already been pointed out by   Niedziałomski in \cite{KN2}. Here we give  a  description  of them.  This is done following  notations for $\Lie G_2$-structures used by Fernández and Gray in \cite{FG} and by the present author in \cite{FMC00, FMC0, FMC}. In our view, the present text illustrates the advantages of such  notations.        
\vspace{2mm} 
  
   Using the map $\xi \lto -\xi \varphi = \nabla \varphi$, where $\varphi$ is the fundamental three-form of the $\Lie G_2$-structure, it is found that all information about the intrinsic torsion is contained in the covariant derivative $\nabla \varphi$. From this it is deduced    that all information about the intrinsic torsion  is contained in the exterior derivatives $d \varphi$ and $d\star \varphi$ ($\star$ denotes Hodge star operator). This is described  in \cite{MMS} (see  also Table I in \cite{FMC0}). Thus a useful alternative way to find such information is by studying $d \varphi$ and $d\star \varphi$. Because all the necessary tools are displayed in first sections, in last section  we take the  chance  to  express the components of  the intrinsic torsion $\xi$ in terms of $d \varphi$ and of  $d \star \varphi$.

\section{Preliminaries}{\indent} \setcounter{equation}{0} \label{preli}
Let $\{ e_0, \dots, e_6 \}$  denote the standard positive orthonormal basis of $\mathbb R^7$. The exceptional Lie group   $\Lie G_2$  is the subgroup of   $\Lie{SO}(7)$ consisting of those elements which fix the three-form $\varphi$  defined  by
\begin{equation}\label{fundphi}
\varphi = \sum_{i \in \mathbb Z_7} e_i^* \w e_{i+1}^* \wedge e_{i+3}^*,
\end{equation}
where we have  denoted by $e_i^*$ the dual one-forms of $e_i$.  A \emph{Cayley  basis} is a $ \{ u_0^*, \dots, u_6^* \}$ for covectors such that, in terms of this basis,   $\varphi$ is written as in \eqref{fundphi}. The name Cayley basis is used  because,  in this context, it is well known the identification of $ \mathbb R^7$ with the imaginary octonions or Cayley numbers $\mathrm{Im} \, \mathbb O$.
  The dual basis, denoted  by $\{u_0, \dots , u_6 \}$ of a Cayley basis,  is also called {\it Cayley basis} for vectors.  Thus the standard basis $\{ e_0, \dots, e_6 \}$  is a Cayley basis.

Here we fix   $\mathrm{Vol} = e_0^* \w \dots \w e_6^*$ as   volume form and, denoting the corresponding Hodge star operator by $\star$, the four-form 
$$
\star \varphi = - \sum_{i \in \mathbb Z_7} e_{i+2}^* \w e_{i+4}^* \wedge e_{i+5}^* \w e_{i+6}^*
$$ 
will play a relevant rôle along the present text.
\vspace{2mm}

The euclidean metric $\langle \cdot , \cdot \rangle$ on $\mathbb R^7$ induces the following inner product on skew symmetric $p$-forms:
$$
(\alpha , \beta) = \frac{1}{p!} \sum_{i_1, \dots , i_p\in \mathbb Z_7} \alpha(e_{i_1}, \dots  , e_{i_p}) \beta(e_{i_1}, \dots  , e_{i_p}),
$$
for all $\alpha, \beta \in \Lambda^p \mathbb R^{7*}$. The corresponding norm of  $(\cdot , \cdot )$ will be denoted by $|\cdot |$.  The inner product $(\cdot , \cdot )$  is widely used in the literature and satisfies the identity 
$$
\alpha \w \star \beta =   (\alpha, \beta) \, \mathrm{Vol}.  
$$
It follows $(\varphi, \varphi) = (\star \varphi, \star \varphi) = 7$ and  $ \mathrm{Vol} = \frac17 \,\varphi\, \w \star \varphi$. Note also that $(\star \alpha, \star \beta) = (\alpha,\beta)$. 
\vspace{2mm}

For general tensors, it is used the inner product $\langle \cdot, \cdot \rangle$  given by 
$$
\langle \Psi , \Phi \rangle  = \sum_{i_1, \dots , i_r,j_1, \dots, j_s  \in \mathbb Z_7} 
 \Psi^{i_1 \dots i_r}_{j_1 \dots j_r}  \Phi^{i_1 \dots i_r}_{j_1 \dots j_r},
$$
for all $(r,s)$-tensors   $\Psi$  and  $\Phi$ on $\mathbb R^7$, where $\Psi^{i_1 \dots i_r}_{j_1 \dots j_r}$ and   $\Phi^{i_1 \dots i_r}_{j_1 \dots j_r}$ are computed by using an orthonormal basis for vectors and its dual one. The corresponding norm of  $\langle \cdot , \cdot \rangle$ will be denoted by $\|\cdot \|$.

\begin{remark}{\rm 
Let us recall the well known  converse fact that if one has  a three-form $\varphi$ on $\mathbb R^7$ and a basis $\{u_0, \dots , u_6\}$ for vectors in $\mathbb R^7$ such that $\varphi$ is expressed  as in \eqref{fundphi}, i.e.
$$
\varphi = \sum_{i \in \mathbb Z_7} u_i^* \w u_{i+1}^* \wedge u_{i+3}^*,
$$
 then the group $\Lie G_2 = \{ a \in \Lie{GL}(7) \, | \, a^* \varphi = \varphi \}$ preserves the metric such  that   $\{u_0, \dots , u_6\}$ is orthonormal and the volume form defined by $\mathrm{Vol} = u_0^*\w \dots \w u_6^*$. It is also denoted by $\Lie G_2$ because it is isomorphic to the one above mentioned. 
}
\end{remark}

Using $\varphi$, a two-fold vector cross product $\times$ is defined by 
$$
(u, v) \longrightarrow u \times v = \sum_{i\in Z_7} \varphi ( u, v, e_i) e_i\in \mathbb R^7
$$    
(\cite{BG},\cite{Gray:vector},\cite{FG}).  This product is skew-symmetric and satisfies 
$$
\langle u \times v, u \rangle =  \langle u\times v, v \rangle=0, \qquad \| u \times v \|^2 = 
\| u  \| \| v \| - \langle u , v \rangle^2,  
$$ 
for all $u,v \in \mathbb R^7$, which are the conditions of the definition of such a product. From these initial conditions,   the following identities are derived 
\begin{eqnarray}
\langle u \times v , w\rangle & = & \langle u , v\times w\rangle,\nonumber\\
u \times (u\times v) & = &  \langle u, v\rangle u - \| u \|^2 v,\label{uu1} \\
u \times (v\times w) & = & -   v \times (u\times w)  + \langle v, w\rangle u +\langle w, u\rangle v  -2 \langle u, v \rangle w \label{uu2}
\end{eqnarray} 

Note that for a Cayley basis, as  $\{ e_0, \dots , e_6\}$, one has
$$
e_i \times e_{i+1} = e_{i+3}, \quad e_{i+3} \times e_{i} = e_{i+1}, \quad e_{i+1} \times e_{i+3} = e_{i},\quad  e_i \times e_i =0, \qquad i \in\mathbb Z_7.
$$

Moreover,    $\Lie G_2$ can be also described as  the subgroup of   $\Lie{SO}(7)$ consisting of those elements which fix the product $\times$, i.e.
$$
\Lie G_2 = \left\{ g \in \Lie{SO}(7) \, | \, g (u \times v ) = g u \times g v, \, \forall \, u, v \in \mathbb R^7\right\}.
$$
As a consequence, the Lie algebra $\lie g_2$ is described by 
$$
\lie g_2 = \left\{ a \in \lie{so}(7) \, | \, a (u \times v ) = a (u) \times v + u \times a(v), \, \forall \, u, v \in \mathbb R^7\right\}.
$$
Even   the following description is also valid  
$$
\lie g_2 = \left\{ a \in \lie{gl}(7) \, | \, a (u \times v ) = a (u) \times v + u \times a(v), \, \forall \,u, v \in \mathbb R^7\right\}.
$$ 
By  considering the identities
$$
a( e_i) = a(e_{i+1}\times e_{i+3}) = a(e_{i+1}) \times e_{i+3} + e_{i+1} \times a( e_{i+1}), \qquad i \in \mathbb Z_7,
$$
and,  denoting $a_{i\, j} = \langle e_i , a(e_j) \rangle$ and so $a= \sum_{i,j \in \mathbb Z_7} a_{i\, j } e_i \otimes e^*_j$, one deduces
$$
\lie g_2 = \left\{ a \in \lie{gl}(7) \, | \, a_{i+1\, i+3} + a_{i+4\, i+5} + a_{i+2\, i+6} =0,\quad a_{i \,j} = - a_{j\, i}  , \, \forall\,  i, j \in \mathbb Z_7\right\}.
$$

Under the action of $\Lie G_2$  given  by $(g a)(v) = g a (g^{-1} v)$, for all $v \in \mathbb R^7$,  the space  $\lie{gl}(7) \cong \mathbb R^{7}\otimes \mathbb R^{7*}$ ($\cong \mathbb R^{7*}\otimes \mathbb R^{7*}$ by using $\langle \cdot, \cdot  \rangle$)   is decomposed into the following irreducible $\Lie G_2$-modules
$$
\lie{gl}(7) = \mathbb R \,\mathrm{Id}_{\mathbb R^7} \oplus S_0^2 \mathbb R^{7*} \oplus \lie g_2 \oplus \lie g_2^\perp,
$$
where 
\begin{eqnarray*}
\mathbb R \,\mathrm{Id}_{\mathbb R^7} & = & \left\{ a \in \lie{gl}(7) \, | \, a = \lambda \,\mathrm{Id}_{\mathbb R^7}, \quad \lambda \in \mathbb R\right\},
\\
S_0^2 \mathbb R^{7*} &= & \left\{ a \in \lie{gl}(7) \, | \, a_{i \, j}   = a_{j \, i}, \; \textstyle \sum_{i\in \mathbb Z_7} a_{i \, i} =0, \; \forall i,j \,\in \mathbb Z_7   \right\},
\\
 \lie g_2^\perp & = & \left\{ a \in \lie{gl}(7) \, | \, a_{i \, j}   = - a_{j \, i}, \;  a_{i+1 \,  i+3} = a_{i+4 \,  i+5} = a_{i+2 \,  i+6},\; \forall\, i,j \in \mathbb Z_7   \right\}.
\end{eqnarray*}
In the literature about $\Lie G_2$-structures, it is frequently denoted $\lie X_1 =  \mathbb R \,\mathrm{Id}_{\mathbb R^7}$, $\lie X_3 = S_0^2 \mathbb R^{7*}$, $\lie X_2 = \lie g_2$ and  $\lie X_4 = \lie g_2^\perp$.  
\vspace{2mm} 

For   $a \in \lie g_2^\perp$, one can consider the vector $v = \sum_{i\in \mathbb  Z_7} v_i e_i$, where $v_i =  a_{i+1 \,  i+3} = a_{i+4 \,  i+5} = a_{i+2 \,  i+6}$. It is found that
$a_{i\, j} = v \lrcorner\varphi(e_i,e_j)$, where $\lrcorner$ denotes the interior product. Thus one has a map $\lie g_2^\perp \longrightarrow \mathbb R^7$, $ a \to v $. Reciprocally, for all $v \in   \mathbb R^7$, one has the endomorphism $A_v$ such that
$$
(A_v)_{i+1\, i+3} = (A_v)_{i+4\, i+5} = (A_v)_{i+2\, i+6} = v_i, \qquad (A_v)_{i\, j} = - (A_v)_{j\, i},
 $$
 for all  $i,j \in \mathbb Z_7$. The map $v \to A_v$ is the converse of the previous one $a \lto v$. Moreover, from the identity $A_v(u) = u \times v$, one easily deduces that $v \to A_v$  is  a $\Lie G_2$-map. Then $\lie g_2^\perp \cong \mathbb R^7$ as $\Lie G_2$-modules.
 \vspace{1mm}
 
 For a general $a \in \lie{gl}(7)$, we will determine  its four $\Lie G_2$-components $a_{(1)}$, $a_{(2)}$, $a_{(3)}$ and  $a_{(4)}$ corresponding to the $\Lie G_2$-modules $\lie X_1$, $\lie X_2$, $\lie X_3$ and $\lie X_4$, respectively. Also the norms of $a_{(i)}$ will  be computed.

 To obtain  $a_{(4)}$, one considers the map $p: \lie{gl}(7)  \lto \mathbb R^7$ (already considered in \cite{RB,KN2}) defined by 
 $$
 p(a) = \textstyle  \sum_{i\in Z_7} (a_{i+1 \,  i+3} - a_{i+3 \,  i+1} + a_{i+4 \,  i+5} - a_{i+5 \,  i+4} + a_{i+2 \,  i+6} - a_{i+6 \,  i+2}) e_i.
 $$    
 We will denote $p(a)_i = a_{i+1 \,  i+3} - a_{i+3 \,  i+1} + a_{i+4 \,  i+5} - a_{i+5 \,  i+4} + a_{i+2 \,  i+6} - a_{i+6 \,  i+2}$.
It follows that, for all $ v  \in \mathbb R^7$, 
$
p(A_v) = 6 v. 
$
Therefore,
\begin{equation}\label{a4}
a_{(4)} = \frac16 A_{p(a)}.
\end{equation}
In a more developed way, $a_{(4)}$ is given by {\small
$$
6a_{(4)}  =    \sumz p(a)_i ( e^{}_{i+1} \otimes e_{i+3}^* - e^{}_{i+3} \otimes e_{i+1}^*
+ e^{}_{i+4} \otimes e_{i+5}^* - e^{}_{i+5} \otimes e_{i+4}^* + e^{}_{i+2} \otimes e_{i+6}^* - e^{}_{i+6} \otimes e_{i+2}^*).
$$  }
 Using \eqref{uu1}, one has 
 \begin{eqnarray*}
 36 \| a_{(4)} \|^2 & = &\textstyle  \sum_{i\in Z_7}  \langle e_i \times p(a) , e_i \times p(a) \rangle = -    \sum_{i\in Z_7}  \langle e_i,   p(a) \times ( p(a) \times e_i) \rangle \\
  & = & - ( \| p(a) \|^2 - 7 \| p(a) \|^2 ) = 6 \| p(a) \|^2.
 \end{eqnarray*} 
  Thus it is obtained 
  $$
  \|   a_{(4)} \|^2 = \frac16 \| p(a) \|^2.
  $$

 It is obvious that 
 \begin{equation}\label{a1}
 a_{(1)} = \frac17 \mathrm{tr}(a)  \mathrm{Id}_{\mathbb R^7}, \qquad \|  a_{(1)}\|^2 = \frac17 (\mathrm{tr}(a))^2,
 \end{equation}
where $\mathrm{tr}(a) = \sum_{i\in \mathbb Z_7}  a_{i \, i}$. 
Taking this into account, one has 
\begin{align}
a_{(3)}= &\textstyle \sumz  (a_{i \, i} - \tfrac17 \mathrm{tr}(a)) e_i \otimes e_i^*
 +\displaystyle  \frac12 \textstyle   \sum_{i, j \in \mathbb Z_7, \, i \ne j} (a_{i \, j} + a_{j \, i}) (e_i \otimes e_j^* +  e_j \otimes e_i^*). \label{a3dos}
 \end{align}
 Therefore,
 \begin{align}
\| (a_{(3)}\|^2  = &\,   \frac12 \textstyle \sum_{i\in \mathbb Z_7} \left(  (a_{i+1 \, i+3} + a_{i+3 \, i+1})^2 +  (a_{i+4 \, i+5} + a_{i+5 \, i+4})^2 +  (a_{i+2 \, i+6} + a_{i+6 \, i+2})^2 \right)\nonumber
\\
 &+ \textstyle  \sum_{i\in \mathbb Z_7}  (  a_{i \, i} - \tfrac17 \mathrm{tr}(a))^2. \label{comp3norm}
\end {align}
 One can check that $\| a_{(1)} + a_{(3)} \|^2 = \| a_{(1)} \|^2  + \|a_{(3)} \|^2$ 
as it is expected.
\vspace{2mm}
 
For $a_{(2)}$, one has 
\begin{eqnarray}
 6 a_{(2)} & = &\textstyle   \sumz (3(a_{i+1\, i+3} - a_{i+3\, i+1}) -  p(a)_i) (e_{i+1} \otimes e_{i+3}^*- e_{i+3} \otimes e_{i+1}^*)\nonumber \\
  & & +\textstyle  \sumz (3(a_{i+4\, i+5} - a_{i+5\, i+4}) -  p(a)_i) (e_{i+4} \otimes e_{i+5}^*- e_{i+5} \otimes e_{i+4}^*) \label{a2uno}\\
 && +  \textstyle  \sumz (3(a_{i+2\, i+6} - a_{i+6\, i+2}) -  p(a)_i) (e_{i+2} \otimes e_{i+6}^*- e_{i+6} \otimes e_{i+2}^*).  \nonumber
\end{eqnarray}
Therefore,
\begin{align}
\| a_{(2)} \|^2  = &  \frac1 2 \textstyle  \sumz \left((a_{i+1\, i+3} - a_{i+3\, i+1})^2+ (a_{i+4\, i+5} - a_{i+5\, i+4})^2 + (a_{i+2\, i+6} - a_{i+6\, i+2})^2\right) \nonumber \\
 & -  \frac1{6}   \|p(a)\|^2. \label{comp2norm}
\end{align}
Note also  that $\| a_{(2)} + a_{(4)}\|^2 = \| a_{(2)} \|^2 + \|a_{(4)}\|^2$ as it is expected. 
\vspace{2mm}

\section{Some invariants  of an endomorphism in  $\lie{gl}(7)$ under the action of $\Lie G_2$}
For our purposes, it is necessary  to consider several invariants (used in \cite{KN2}) of an endomorphism $a \in  \lie{gl}(7)$ under the action of $\Lie G_2$. There  are some of them which are also  invariant under  the wider  action of  $\Lie{SO}(7)$.

A first one is the norm of $a$. In fact, $\| a \|^2 = \| g a \|^2$, for all $g \in  \Lie{SO}(7)$. Other invariants are the coefficients  of the characteristic polynomial. They are denoted by $\sigma_i(a)$, where
$$
\det (a - \lambda \, \mathrm{Id}) = \textstyle  \sum_{i=1}^7 (-1)^i \sigma_{7-i}(a) \lambda^i.
$$ 
Thus, $\sigma_0(a) =1$,  $\sigma_1(a) = \tr (a)$, $\sigma_2(a) = \tfrac12\sum_{i,j \in \mathbb Z_7 \, i \ne j} (a_{i\, i} a_{j \, j} - a_{i\, j} a_{j \, i} )$.
 \vspace{2mm} 
  
In relation with the $\Lie G_2$-components of $a$, one has the following  result.
\begin{lemma} \label{lesa1} For all $a \in \lie{gl}(7)$, it is satisfied 
\begin{eqnarray*}
 (\sigma_1(a))^2 & = &  7 \| a_{(1)} \|^2,\\
 2 \sigma_2(a) & = & 6  \| a_{(1)} \|^2  +  \| a_{(2)} \|^2 -  \| a_{(3)} \|^2+    \| a_{(4)} \|^2.
\end{eqnarray*}
\end{lemma}
\begin{proof}
The expression for $ (\sigma_1(a))^2$  follows directly from \eqref{a1}. To see the expression for $\sigma_2(a)$, taking
\begin{eqnarray*}
\| a_{(2)} + a_{(4)}\|^2 &= & \frac12\textstyle  \sum_{i,j \in \mathbb Z_7 \, i \ne j} (a_{i\,j}^2 + a_{j\,i}^2) 
-\displaystyle  \frac12\textstyle    \sum_{i,j \in \mathbb Z_7 \, i \ne j} a_{i\,j} a_{j\,i},
\\
\| a_{(1)} + a_{(3)}\|^2 & = & \frac12 \textstyle \sum_{i,j \in \mathbb Z_7 \, i \ne j} (a_{i\,j}^2 + a_{j\,i}^2) +  \displaystyle \frac12 \textstyle  \sum_{i,j \in \mathbb Z_7 \, i \ne j} a_{i\,j} a_{j\,i} + \sumz a_{i\,i}^2 
\end{eqnarray*}
 into account,    it is obtained  
$$
\textstyle \| a_{(2)}\|^2  + \|a_{(4)}\|^2- \| a_{(1)}\|^2  -  \|a_{(3)}\|^2 = -   \sum_{i,j \in \mathbb Z_7 \, i \ne j} a_{i\,j} a_{j\,i} -   \sumz a_{i\,i}^2. 
$$
Since $\sigma_1(a)^2 =  \sumz a_{i\,i}^2 +  \sum_{i,j \in \mathbb Z_7 \, i \ne j} a_{i\,i} a_{j\,j}$, one deduces
$$
\| a_{(2)}\|^2  + \|a_{(4)}\|^2- \| a_{(1)}\|^2  -  \|a_{(3)}\|^2 =  2 \sigma_2(a) - \sigma_1(a)^2. 
$$
Finally, using  $\sigma_1(a)^2 = 7  \| a_{(1)}\|^2$, the identity for $\sigma_2(a)$ is obtained.
\end{proof}
\vspace{2mm}

Now, let us  restrict our attention  to the action of $\Lie G_2$. There are  the following  invariants associated with an endomorphism $a$:
\begin{eqnarray*}
i_0(a) & = &\textstyle  \sum_{i,j \in \mathbb Z_7} \langle a(e_i) \times  a(e_j), e_i \times e_j \rangle,\\
i_1(a) & = & \textstyle \sum_{i,j \in \mathbb Z_7} \langle a(e_i) \times e_i, a(e_j)  \times e_j \rangle,\\
i_2(a) & = &\textstyle  \sum_{i,j \in \mathbb Z_7} \langle a(e_i) \times e_j, a(e_j)  \times e_i \rangle.
 \end{eqnarray*}
\begin{lemma} \label{leia2}For all $a \in \lie{gl}(7)$, it is satisfied 
\begin{eqnarray*}
i_0(a) & = &  6 \|a_{(1)} \|^2 +  3 \|a_{(2)} \|^2 -  \|a_{(3)} \|^2-  3 \|a_{(4)} \|^2, \\
i_1(a)  &  = & 6 \|a_{(4)} \|^2,\\
i_2(a )& = &- 6  \|a_{(1)} \|^2 +  3 \|a_{(2)} \|^2 +  \|a_{(3)} \|^2 - 3 \|a_{(4)} \|^2,\\
6 \| a \|^2 & = & \textstyle \sum_{i , j \in \mathbb Z_7} \| e_i \times a(e_j) \|^2.
\end{eqnarray*}
\end{lemma}
\begin{proof}
For $i_1(a)$, one has the identity
\begin{align*}
i_1(a)  = &\textstyle  \sum_{i \in \mathbb Z_7} \langle a(e_i) \times e_i, a(e_i)  \times e_i \rangle
 + 2 \sum_{i \in \mathbb Z_7} \langle a(e_{i+1}) \times e_{i+1}, a(e_{i+3})  \times e_{i+3} \rangle\\
 & +2\textstyle  \sum_{i \in \mathbb Z_7} \langle a(e_{i+4}) \times e_{i+4}, a(e_{i+5})  \times e_{i+5} \rangle
 +2 \sum_{i \in \mathbb Z_7} \langle a(e_{i+2}) \times e_{i+2}, a(e_{i+6})  \times e_{i+6} \rangle.
\end{align*} 
Now we will compute  all these summands. By one hand, one has 
$$
\textstyle \sum_{i \in \mathbb Z_7} \langle a(e_i) \times e_i, a(e_i)  \times e_i \rangle = - \sum_{i \in \mathbb Z_7}\langle a(e_i),  e_i \times (e_i \times a(e_i)) \rangle.
$$
Therefore, by using \eqref{uu1}, it is obtained
$$
\textstyle \sum_{i \in \mathbb Z_7} \langle a(e_i) \times e_i, a(e_i)  \times e_i \rangle =   \| a \|^2 - \sum_{i \in \mathbb Z_7} a_{i\,i}^2. 
$$
By the another hand, one has 
\begin{eqnarray*}
\textstyle  \sum_{i \in \mathbb Z_7} \langle a(e_{i+1}) \times e_{i+1}, a(e_{i+3})  \times e_{i+3} \rangle & =&  - \textstyle \sumz a_{i+1\,i+3}  a_{i+3\,i+1}  - \sumz  a_{i+3\,i+1}   a_{i+2\,i+6} \\
 & & -\textstyle \sumz  a_{i+5\,i+4}   a_{i+2\,i+6} 
 +  \textstyle  \sumz  a_{i+1\,i+3}   a_{i+4\,i+5} \\
 && +\textstyle  \sumz  a_{i+5\,i+4}   a_{i+6\,i+2},
\end{eqnarray*}
\begin{eqnarray*}
\textstyle  \sum_{i \in \mathbb Z_7} \langle a(e_{i+4}) \times e_{i+4}, a(e_{i+5})  \times e_{i+5} \rangle & =&  - \textstyle \sumz a_{i+4\,i+5}  a_{i+5\,i+4}  - \sumz  a_{i+1\,i+3}   a_{i+6\,i+2} \\
 & & + \textstyle \sumz  a_{i+4\,i+5}   a_{i+2\,i+6} 
 +   \textstyle \sumz  a_{i+3\,i+1}   a_{i+6\,i+2} \\
 && -\textstyle  \sumz  a_{i+1\,i+3}   a_{i+4\,i+5},
\end{eqnarray*}
\begin{eqnarray*}
\textstyle  \sum_{i \in \mathbb Z_7} \langle a(e_{i+2}) \times e_{i+2}, a(e_{i+6})  \times e_{i+6} \rangle & =&  -\textstyle  \sumz a_{i+2\,i+6}  a_{i+6\,i+2}  - \sumz  a_{i+3\,i+1}   a_{i+4\,i+5} \\
 & & +\textstyle  \sumz  a_{i+1\,i+3}   a_{i+2\,i+6} 
 -  \textstyle  \sumz  a_{i+4\,i+5}   a_{i+6\,i+2} \\
 && - \textstyle \sumz  a_{i+3\,i+1}   a_{i+5\,i+4}.
\end{eqnarray*}
From these identities, it follows
\begin{eqnarray*}
i_1(a) & = & \textstyle \sumz (a_{i+1\,i+3} -a_{i+3\,i+1})^2 +  \sumz (a_{i+4\,i+5} -a_{i+5\,i+4})^2
+ \sumz (a_{i+2\,i+6} -a_{i+6\,i+2})^2\\
& & + 2 \textstyle  \sumz (a_{i+1\,i+3} -a_{i+3\,i+1}) (a_{i+4\,i+5} -a_{i+5\,i+4})\\
&&  +  2 \textstyle  \sumz (a_{i+1\,i+3} -a_{i+3\,i+1}) (a_{i+2\,i+6} -a_{i+6\,i+2})\\
& &  +  2 \textstyle   \sumz (a_{i+4\,i+5} -a_{i+5\,i+4}) (a_{i+2\,i+6} -a_{i+6\,i+2})\\
& = & \| p(a) \|^2 = 6 \| a_{(4)} \|^2.
  \end{eqnarray*} 
For $i_2(a)$, one has the identity
\begin{align*}
i_2(a)  = &\textstyle  \sum_{i \in \mathbb Z_7} \langle a(e_i) \times e_i, a(e_i)  \times e_i \rangle
 + 2 \sum_{i \in \mathbb Z_7} \langle a(e_{i+1}) \times e_{i+3}, a(e_{i+3})  \times e_{i+1} \rangle\\
 & +2 \textstyle \sum_{i \in \mathbb Z_7} \langle a(e_{i+4}) \times e_{i+5}, a(e_{i+5})  \times e_{i+4} \rangle
 +2 \sum_{i \in \mathbb Z_7} \langle a(e_{i+2}) \times e_{i+6}, a(e_{i+6})  \times e_{i+2} \rangle.
\end{align*} 
The first sum was already computed above. The following three sums are given by   
\begin{align*}
\textstyle \sum_{i \in \mathbb Z_7} \langle a(e_{i+1}) \times e_{i+3}, a(e_{i+3})  \times e_{i+1} \rangle  = & - \textstyle \sumz a_{i+1\,i+1}  a_{i+3\,i+3} - \sumz a_{i+5\,i+4}  a_{i+6\,i+2} \\
 & -\textstyle  \sumz a_{i+1\,i+3}  a_{i+4\,i+5} + \sumz a_{i+5\,i+4}  a_{i+2\,i+6} \\
& + \textstyle \sumz a_{i+3\,i+1}  a_{i+2\,i+6},        
\end{align*} 
\begin{align*}
\textstyle \sum_{i \in \mathbb Z_7} \langle a(e_{i+4}) \times e_{i+5}, a(e_{i+5})  \times e_{i+4} \rangle  = & -\textstyle \sumz a_{i+4\,i+4}  a_{i+5\,i+5} - \sumz a_{i+3\,i+1}  a_{i+6\,i+2} \\
 & +\textstyle  \sumz a_{i+1\,i+3}  a_{i+5\,i+4} - \sumz a_{i+4\,i+5}  a_{i+2\,i+6} \\
& +\textstyle  \sumz a_{i+1\,i+3}  a_{i+6\,i+2},        
\end{align*} 
\begin{align*}
\textstyle  \sum_{i \in \mathbb Z_7} \langle a(e_{i+2}) \times e_{i+6}, a(e_{i+6})  \times e_{i+2} \rangle  = & - \textstyle \sumz a_{i+2\,i+2}  a_{i+6\,i+6} - \sumz a_{i+1\,i+3}  a_{i+6\,i+2} \\
 & +\textstyle  \sumz a_{i+3\,i+1}  a_{i+4\,i+5} - \sumz a_{i+3\,i+1}  a_{i+4\,i+5} \\
& +\textstyle  \sumz a_{i+4\,i+5}  a_{i+6\,i+2}.        
\end{align*} 
Therefore,
\begin{eqnarray*}
i_2(a) & = &\textstyle \sumz \left(  (a_{i+1\,i+3}^2 + a_{i+3\,i+1}^2 +   a_{i+4\,i+5}^2 + a_{i+5\,i+4}^2
+  a_{i+2\,i+6}^2 +a_{i+6\,i+2}^2\right)   \\
& & - 2  \textstyle \sumz \left(a_{i+1\,i+1}  a_{i+3\,i+3}+a_{i+4\,i+4} a_{i+5\,i+5} + a_{i+2\,i+2} a_{i+6\,i+6} \right)\\
& & - 2\textstyle  \sumz ( a_{i+1\,i+3} - a_{i+3\,i+1})  ( a_{i+4\,i+5} - a_{i+5\,i+4})  \\
& & - 2 \textstyle  \sumz ( a_{i+1\,i+3} - a_{i+3\,i+1})  ( a_{i+2\,i+6} - a_{i+6\,i+2})  \\
& & - 2 \textstyle \sumz ( a_{i+4\,i+5} - a_{i+5\,i+4})  ( a_{i+2\,i+6} - a_{i+6\,i+2}).
\end{eqnarray*} 
From this, it follows
\begin{eqnarray*}
i_2(a) & = & \| a\|^2 - 7 \| a_{(1)} \|^2 + 2 \| a_{(2)} + a_{(4)}\|^2 - 6  \|a_{(4)}\|^2\\
           & = &  - 6 \| a_{(1)} \|^2 + \| a_{(3)} \|^2 +3 \| a_{(2)} \|^2 - 3 \|a_{(4)}\|^2.
\end{eqnarray*} 

For $i_0(a)$,  using \eqref{uu2}, one has  
\begin{eqnarray*}
i_0(a) & = &- \textstyle  \sum_{i,j \in \mathbb Z_7} \langle a(e_i),   a(e_j) \times (e_j \times e_i) \rangle\\
& =& \textstyle  \sum_{i,j \in \mathbb Z_7}   \langle a(e_i) \times e_j,   a(e_j) \times e_i\rangle
- \| a\|^2 -  \sum_{i,j \in \mathbb Z_7}  a_{i\, j} a_{j\, i} + 2   \sum_{i,j \in \mathbb Z_7}  a_{i\, i} a_{j\, j}.
\end{eqnarray*}
This identity can be writen in the following way 
 \begin{eqnarray*}
i_0(a) & = & i_2(a)  - \| a\|^2  + 2 \sigma_2(a) - \textstyle \sum_{i \in \mathbb Z_7}  a_{i\, i}^2 +   2  \sum_{i \in \mathbb Z_7}  a_{i\, i}^2 +  \sum_{i,j \in \mathbb Z_7 \, i\ne j }  a_{i\, i} a_{j\, j}\\
& = &  i_2(a)  - \| a\|^2  + 2 \sigma_2(a) + \sigma_{1}(a)^2
\end{eqnarray*}
(note that this is one identity given in \cite{KN2} and displayed below).  Finally, using the expressions  for $ i_2(a)$,  $\sigma_2(a)$ and $\sigma_{1}(a)^2$, it follows
 \begin{eqnarray*}
  i_0(a)  & = & -6  \|a_{(1)} \|^2 +  3 \|a_{(2)} \|^2 +  \|a_{(3)} \|^2 - 3 \|a_{(4)} \|^2  - \| a\|^2\\
          & &+  6  \| a_{(1)} \|^2  +  \| a_{(2)} \|^2 -  \| a_{(3)} \|^2+    \| a_{(4)} \|^2 +  7  \| a_{(1)} \|^2 \\
          &=&   6  \| a_{(1)} \|^2 +  3 \|a_{(2)} \|^2   -  \|a_{(3)} \|^2 - 3 \|a_{(4)} \|^2. 
\end{eqnarray*}

The fourth identity in Lemma follows by using \eqref{uu1} in 
$$
\textstyle \sum_{i,j \in \mathbb Z_7} \|e_i \times a(e_j) \|^2 =  - \sum_{i,j\in \mathbb Z_7} \langle a(e_j), e_i \times (e_i \times a(e_j)) \rangle.  
$$
\end{proof}

From Lemma \ref{lesa1} and  Lemma \ref{leia2}, it is direct to check the identities proved by Niedzia\l omski.  
\begin{lemma}[\cite{KN2}] \label{leia3}. 
For all $a \in \lie{gl}(7)$, the following relations hold:
\begin{eqnarray*}
i_1(a) & = & - i_0(a) + \| a\|^2 + 4 \sigma_2(a) - \sigma_1(a)^2,\\
i_2(a) & = & i_0(a) + \|a\|^2 - 2 \sigma_2(a) - \sigma_1(a)^2.
\end{eqnarray*}
In particular, 
$$
i_1(a) - i_2(a) = -2 i_0(a) + 6 \sigma_2(a).
$$
\end{lemma}

Finally, it is interesting to note that, reciprocally, the norms of the $\Lie G_2$-components of an endomorphism $a$ can be expressed in terms of its invariants. In fact, the identities in next lemma  easily follow  from Lemma \ref{lesa1},  Lemma \ref{leia2} and  Lemma \ref{leia3}. 
\begin{lemma} For all $a \in \lie{gl} (7)$, it is satisfied
\begin{eqnarray*} 
\| a_{(1)} \|^2 & = & \sigma_1(a)^2,  \qquad   a_{(1)} = \tfrac17  \sigma_1(a) \textstyle\sumz e_i^* \otimes e_i,\\
6 \| a_{(2)}\|^2 & = & i_0(a) + i_1(a) + i_2(a)\, = \;- 6 \sigma_2(a)  +  3 i_0(a) + 2 i_1(a),\\
6 \| a_{(3)}\|^2 & =  & 12 \, \sigma_1(a) ^2 - i_0(a)  + i_2(a)\; = \; 12 \sigma_1(a) ^2- 6\sigma_2(a) + i_0(a) + i_1(a),\\  
6 \| a_{(4)}\|^2 & = &   i_1(a).
\end{eqnarray*}
\end{lemma}  
\vspace{2mm}

\section{$\Lie{G}_2$-structures} {\indent}
\setcounter{equation}{0} 
   We recall briefly some facts about ${\sl G}_2$-structures and
display some results we will need later.

A ${\sl G}_2$-structure on a Riemannian seven-manifold
$(M,\langle\cdot ,\cdot \rangle)$ is by 
definition a 
$\Lie{G}_2$-reduction $\mathcal P(M)$  of the orthonormal frame $\Lie{SO}(7)$-bundle $\mathcal{SO}(M)$. This is equivalent to the existence of a global three-form
$\varphi$ which may be locally written as in (\ref{fundphi}). For all
$m \in M$, the
tangent space $T\,M$ is then associated to the representation $ \mathbb R^7$ of $\Lie{G}_2$.
Hence, in each point of a Riemannian manifold with a  ${\sl
G}_2$-structure,  there are  local orthonormal coframe fields
$\{ \- e_0^*, \- \ldots, \- e_6^* \}$ such that
$\varphi$ is written as in (\ref{fundphi}), we call them {\it Cayley
coframes}.  The dual frame, denoted also by $\{
e_0, \dots , e_6 \}$,  of a Cayley coframe is called  a {\it Cayley frame}.  

When $M$ is equipped with a $\Lie{G}_2$-structure, one  can also consider the two-fold vector cross product   $\times : T M \times T M \to T
M$ given by $\langle X \times Y , Z \rangle = \varphi (X,Y,Z)$ \cite{FG}. 

If $M$ is equipped with a
$\Lie{G}_2$-structure, then there always  exists a $\Lie{G}_2$-connection
$\widetilde{\nabla}$ defined on $M$. Doing the difference
$\widetilde{\xi}_X = \widetilde{\nabla}_X - \nabla_X$, where
$\nabla_X$ is the Levi-Civita connection of $\langle \cdot , \cdot
\rangle$, a tensor $\widetilde{\xi}_X \in \lie{so}(M)$ is
obtained.  Decomposing $\widetilde{\xi}_X = ( \widetilde{\xi}_X
)_{\lie{g}_{2}} + ( \widetilde{\xi}_X )_{\lie{g}_{2}^\perp}$,
$( \widetilde{\xi}_X )_{\lie{g}_{2}} \in \lie{g}_{2}$
and $( \widetilde{\xi}_X )_{\lie{g}_{2}^\perp} \in
\lie{g}_{2}^\perp$, a new $\Lie{G}_2$-connection $\nabla^{\Lie G_2}$, defined by
$\nabla^{\Lie G_2}_X = \widetilde{\nabla}_X - (\tilde{\xi}_X
)_{\lie{g}_{2}}$, can be considered. Because the difference
between two $\Lie G_2$-connections must be in $\lie{g}_{2}$,
$\nabla^{\Lie G_2}$ is the unique $\Lie G_2$-connection on $M$ such that its
torsion $\xi_X = ( \widetilde{\xi}_X
)_{\lie{g}_{2}^\perp } = \nabla^{\Lie G_2}_X - \nabla_X$ is in
$\lie{g}_{2}^\perp$. $\nabla^{\Lie G_2}$ is called the {\it minimal
connection} and $\xi$ is referred to as the {\it intrinsic
torsion} of the $\Lie G_2$-structure on $M$
\cite{CleytonSwann:torsion}. 
\vspace{1mm}

By one hand, as it is pointed out in \cite{KN2}, since $\xi_X \in \lie{g}_{2}^\perp$, one has 
$$
\xi_X Y = A_{T(X)} (Y) = Y \times T(X),
$$
where $T\in \mathrm T M \otimes \mathrm T^* M$  (see the description of $\lie{g}_{2}^\perp$ given in Section \ref{preli}).

On the other hand, the covariant derivative $\nabla \varphi$ satisfies $\nabla \varphi = - \xi \varphi$ and $\xi \to - \xi \varphi$ is a $\Lie{G}_2$-map. Thus the space of   possible intrinsic torsions can be identified with the space $\mathfrak X \subseteq \mathrm T^* M \otimes \Lambda^3  \mathrm T^* M  $ of covariant derivatives  of $\varphi$. Consider now the $\Lie G_2$-map $r : \mathfrak X \longrightarrow   \mathrm T^* M \times    \mathrm T^* M$ defined by \cite{FMC0, FMC00} 
$$
r(\alpha) (X,Y) = \frac14  (X\lrcorner \alpha , Y\lrcorner \star \varphi ) = \frac1{24}  \langle X\lrcorner \alpha , Y\lrcorner \star \varphi \rangle.
$$   
Then an explicit expression \cite{FMC} for $\xi$  is given by 
$$
\xi_X Y = - \frac13 \textstyle  \sumz r(\nabla \varphi) (X,e_i) e_i \times Y.
$$
Thus, comparing with the identity $\xi_X Y = Y \times T(X)$, it is found \cite{KN2}
$$
T(X) = \frac13 \textstyle  \sumz r(\nabla \varphi) (X,e_i) e_i.
$$
Note that $T_{i \, j} = \langle e_i, T(e_j) \rangle = \frac13 r(\nabla \varphi) (e_j,e_i)$.
\vspace{2mm}

The space $\mathfrak X = \mathrm T^* M \otimes \lie g_2^\perp  \subseteq \mathrm T^* M \otimes (\mathrm T M \otimes \mathrm T^* M)$ of posible  intrinsic torsions, under the action of $\Lie G_2$,  is decomposed into four irreducible $\Lie G_2$-modules $\lie X = \lie X_1 \oplus \lie X_2 \oplus     \lie X_3 \oplus  \lie X_4$. This gives place to a natural classification of types of $\Lie G_2$-structure. This was shown by Fernández and Gray, by using $\nabla \varphi$, in \cite{FG}.
 Using  the composition of maps $\xi \to - \xi \varphi = \nabla \varphi \to r(-\xi \varphi) \to T$, it is derived $\lie X \cong  \lie{gl}(\mathrm T M) $,
$\lie X_1 \cong \mathbb R \,\mathrm{Id}_{\mathrm T M}$, $\lie X_3 \cong S_0^2 \mathrm T^* M$, $\lie X_2 \cong \lie g_2$, and $\lie X_4 \cong \lie g_2^\perp$.
\vspace{2mm}

\section{The different  integral formulas }
By using different tools, some authors have derived integral formulas for $\Lie G_2$-structures. Our purpose here is to display  the correlations between them. The existence of theses correlations was indicated by Kiedzialomski in \cite{KN2}.  Firstly, it is shown   the integral formula more recently proved.
\begin{theorem}[\cite{KN2}]
On a closed (i.e., compact without boundary) manifold $M$ equipped with a $\Lie G_2$-structure, the following integral formula holds
$$
\frac16 \int_M s \, \mathrm{vol}_M =  \int_M - \frac32  i_0(T)  + 6 \sigma_2(T) \; \mathrm{vol}_M, 
$$
where $s$ denotes the scalar curvature. 
\end{theorem}    
This is shown by deducing   the divergence formula given by 
\begin{equation} \label{divkamil}
\textstyle  \mathrm{div} \sumz \xi_{e_i} e_i  =\displaystyle -  \frac16 s  -  \frac32 i_0(T) + 6 \sigma_2(T)  
\end{equation} 
(see Lemma 2.2, Lemma 4.1 and  Lemma 4.2 in \cite{KN2}). For the vector field $\sumz \xi_{e_i} e_i$, one has 
\begin{equation} \label{peT}
\textstyle  \sumz \xi_{e_i} e_i =\displaystyle  - p(T) =- \frac16 (pd^\star\varphi)^\sharp,
\end{equation}
where $pd^\star\varphi = \star (\star d\varphi \wedge \varphi) = - \star (\star d \star \varphi \w \star \varphi)$ \cite{FMC0, FMC00} and we  have used the musical isomorphism $\langle \alpha^\sharp, x\rangle = \alpha(x)$, for $\alpha \in \mathrm{T}^*_m M$, $x\in \mathrm T_m M$ and $m\in M$.  In fact, considering a Cayley frame $\{e_0, \dots, e_6\}$ (this frame will be also used as Cayley frame in the sequel),  
\begin{align*}
\textstyle \sumz \xi_{e_i} e_i  = & \textstyle  \sumz e_i \times T(e_i) \\
 = &\textstyle  \sumz \left( T_{i +1\, i} \,  e_i \times e_{i+1} +  T_{i +2\, i} \,  e_i \times e_{i+2} +T_{i +3\, i} \,  e_i \times e_{i+3}   + T_{i +4\, i} \,  e_i \times e_{i+4}\right. \\
&\qquad   \left. + T_{i +5\, i} \,  e_i \times e_{i+5} +  T_{i +6\, i} \,  e_i \times e_{i+6} \right) \\
 = & \textstyle \sumz \left( T_{i +1\, i} e_{i+3} +  T_{i +2\, i} e_{i+6} - T_{i +3\, i} e_{i+1}+   T_{i +4\, i} e_{i+5} -  T_{i +5\, i} e_{i+4} -   T_{i +6\, i} e_{i+2}\right) \\
   = & -\textstyle  \sumz  \left(   T_{i +1\, i+3 } - T_{i +3\, i+1} +  T_{i +4\, i+5} - T_{i +5\, i+4} + T_{i +2\, i+6} - T_{i +6\, i+2}\right) e_{i} \\ç
   = & - p(T).
\end{align*}
Finally, note that  $p(T) = - \frac13 p(r(\nabla \varphi))= \frac16 (pd^\star\varphi)^\sharp$ 
(see \cite{FMC}). 
\vspace{3mm}

Now, taking Lemma \ref{lesa1} and \ref{leia2} into account, it is obtained 
the following divergence equation and 
integral formula in  terms of the $\Lie G_2$-components of $T$
\begin{equation}\label{divformTT}
\frac16 s \, =    9  \|T_{(1)} \|^2 - \frac32 \|T_{(2)} \|^2 - \frac32 \|T_{(3)} \|^2+ \frac{15}2 \|T_{(4)} \|^2 + \mathrm{div} \, (p(T)). 
\end{equation}
As consequence, one has the following result.
\begin{theorem}
On a closed  manifold $M$ equipped with a $\Lie G_2$-structure, the following integral formula holds
\begin{equation}\label{formTT}
\frac16 \int_M s \, \mathrm{vol}_M =  \int_M  9  \|T_{(1)} \|^2 - \frac32 \|T_{(2)} \|^2 - \frac32 \|T_{(3)} \|^2+ \frac{15}2 \|T_{(4)} \|^2 \; \mathrm{vol}_M. 
\end{equation}
\end{theorem} 
By the fourth identity given in Lemma \ref{leia2}, it is obtained  $\|\xi\|^2 = 6 \| T\|^2$. Taking this and $p(T) = -  \sumz \xi_{e_i} e_i$ into account in follows next divergence equation 
$$
   s = 9 \|\xi_{(1)} \|^2  - \frac32 \|\xi_{(2)} \|^2  - \frac32 \|\xi_{(3)} \|^2+  \frac{15}2 \|\xi_{(4)} \|^2 - 6\,  \mathrm{div} \, \textstyle \sumz \xi_{e_i} e_i.
$$
Hence  one has next consequence.  
 \begin{corollary}
On a closed  manifold $M$ equipped with a $\Lie G_2$-structure with intrinsic torsion $\xi$, the following integral formula holds
\begin{equation} \label{herefor}
 \int_M s \, \mathrm{vol}_M =  \int_M  9 \|\xi_{(1)} \|^2  - \frac32 \|\xi_{(2)} \|^2  - \frac32 \|\xi_{(3)} \|^2+  \frac{15}2 \|\xi_{(4)} \|^2 \; \mathrm{vol}_M.
\end{equation} 
\end{corollary}
Now we  will compare this identity with the integral formula given by Bor and Hernández Lamoneda.
\begin{theorem}[\cite{BHLS1}]
On a closed  manifold $M$ equipped with a $\Lie G_2$-structure, the following integral formula holds
\begin{equation} \label{BorHL}
 \int_M s \, \mathrm{vol}_M =  \int_M 9 \| (\nabla \varphi)_{(1)}\|^2 - \frac32  \| (\nabla \varphi)_{(2)}\|^2 -\frac32  \| (\nabla \varphi)_{(3)}\|^2   + \frac{15}2 \| (\nabla \varphi)_{(4)}\|^2 \; \mathrm{vol}_M.
\end{equation}
\end{theorem} 

Note the coincidence relative to the coeffitient numbers.  However, $\| (\nabla \varphi)_{(i)}\|^2$ is not equal to $\| \xi_{(i)}\|^2$. What is  happening?. The answer  is based in the conventions followed in \cite{BHLS1} for the exterior product, the exterior derivative, etc. Such  conventions are those ones  fixed in \cite{KobNom}. Here we follow those ones fixed, for instance, in \cite{Boothby,ON}. Because of this, one has 
$
\varphi = 3! \widetilde \varphi, 
$ 
where $\widetilde \varphi$ is the three-form considered in \cite{BHLS1}. Below we will prove that $\| \nabla \varphi \|^2 = 36 \| \xi \|^2$. Therefore, 
$ \| \nabla \widetilde \varphi \|^2 =  \frac1{36} \| \nabla \varphi \|^2 = \| \xi \|^2$ and formulas \eqref{herefor} and \eqref{BorHL} are agree. Thus, Bor and Hernández Lamoneda are really right when they  say intrinsic torsion in the introduction of  \cite{BHLS1}.
\vspace{2mm}

Taking $3\, T_{j \,i} = r(\nabla \varphi) (e_i , e_j)$ and Lemma  2.1 of \cite{FMC0} into account,    $\nabla \varphi$ can be expressed as 
$$
\nabla \varphi = 3 \textstyle  \sum_{i,j \in \mathbb Z_7 } T_{j\, i}  e_i^* \otimes e_j \lrcorner \star \varphi.
$$
 This implies $\| \nabla \varphi \|^2= 9 \, . \, 4 \, . \, 6\,\|T \|^2 = 36\, \| \xi \|^2$ as it was  claimed.
\vspace{2mm}

 In \cite{RB}, Bryant has also given an integral formula for $\Lie G_2$-structures. This is expressed  in terms of  $\tau_0$, $\tau_1$, $\tau_2$ and  $\tau_3$ which are defined by 
\begin{align*}
d\varphi &=\tau_0\, \underline \star \, \varphi+3\tau_1\wedge\varphi+\underline  \star \,\tau_3, \\
d\,\underline \star \, \varphi &=4\tau_1\wedge  \underline \star \,\varphi+\tau_2\wedge\varphi.
\end{align*} 
where $\underline \star$ is the Hodge star operator induced by the volume form $\underline{\mathrm{vol}}_M$.   The function  $\tau_0$  determines $\xi_{(1)}$, the one-form  $\tau_1$ determines $\xi_{(4)}$, the two-form $\tau_2\in \Lambda^2_{(14)}T^{\ast}M\cong \lie g_2$ determines $\xi_{(2)}$ and the three-form  $\tau_3\in\Lambda^3_{(27)}T^{\ast}M\cong S_0^2 (T^{\ast}M) $ determines $\xi_{(3)}$.   The volume form  $\underline{\mathrm{vol}}_M$, fixed by Bryant, is such that $\underline{\mathrm{vol}}_M =  - \mathrm{vol}_M$. Hence $\underline \star = - \star$. But this is not a problem, because the corresponding norms are the same. Hence   we will write
\begin{align*}
d\varphi &=\tau_0\, \star \, \varphi+3\tau_1\wedge\varphi+  \star \,\tau_3, \\
d\, \star \, \varphi &=4\tau_1\wedge   \star \,\varphi+\tau_2\wedge\varphi,
\end{align*}
and the integral formula given by Bryant is displayed  in next result.  
\begin{theorem}[\cite{BHLS1}]
On a closed  manifold $M$ equipped with a $\Lie G_2$-structure, the following integral formula holds
\begin{equation} \label{eq:Bryant1}
\int_M s\,{\rm vol}_M=\int_M \frac{21}{8}\tau_0^2+30|\tau_1|^2-\frac{1}{2}|\tau_2|^2-\frac{1}{2}|\tau_3|^2\,{\rm vol}_M.
\end{equation}\label{divBryant}
\end{theorem}
It is denoted $| \cdot |$, because it is used $(\cdot , \cdot)$. Likewise, the corresponding divergence formula is given by  
\begin{equation}\label{eq:Bryant2}
s=12 \, d^\star \tau_1+\frac{21}{8}\tau_0^2+30|\tau_1|^2-\frac{1}{2}|\tau_2|^2-\frac{1}{2}|\tau_3|^2.
\end{equation}
Note that, despite of different volume forms, the one-form $\tau_1$ will be the same and,  for  the coderivatives, one has $d^\star = d^{\underline \star}$.  We recall that  $d^\star  \tau_1 = - \mathrm{div}\, \tau_1^\sharp$. 
\begin{lemma}
The following identities hold
$$
\tau_0^2 = \frac{12^2}{7^2} \sigma_1(T)^2 = \frac{12^2}{7} \| T_{(1)}\|^2, \qquad |\tau_1|^2 = \frac14 \|p(T)\|^2 = \frac32 \|T_{(4)}\|^2,
$$
$$
|  \tau_3 |^2  = 18 \| T_{(3)} \|^2, \qquad |  \tau_2 |^2 = 18 \| T_{(2)} \|^2.
$$
\end{lemma}
\begin{proof}
In \cite{FMC0}, the exterior derivative $d \varphi$ is expressed as 
\begin{align}
 d\varphi  = & - 3 \textstyle \sum_{i\in \mathbb Z_7} (T_{i+2\,i+2} + T_{i+4\,i+4} + T_{i+5\,i+5} + T_{i+6\,i+6})  e_{i+2}^* \wedge e_{i+4}^* \wedge e_{i+5}^* \wedge e_{i+6}^*  \nonumber\\
&  +  3\textstyle  \sum_{i\in \mathbb Z_7} (  T_{i+3\,i+1} + T_{i+5\,i+4}+ T_{i+6\,i+2} ) e_{i}^* \wedge  e_{i+1}^* \wedge e_{i+2}^* \wedge e_{i+4}^* \nonumber\\
&  +  3\textstyle \sum_{i\in \mathbb Z_7} (- T_{i+1\,i+3} - T_{i+4\,i+5} + T_{i+6\,i+2} )  e_{i}^* \wedge   e_{i+2}^* \wedge e_{i+3}^* \wedge e_{i+5}^*\label{extder}\\
&  +  3\textstyle \sum_{i\in \mathbb Z_7} (- T_{i+1\,i+3} +T_{i+5\,i+4} - T_{i+2\,i+6} )  e_{i}^* \wedge   e_{i+3}^* \wedge e_{i+4}^* \wedge e_{i+6}^*\nonumber \\
&  + 3  \textstyle \sum_{i\in \mathbb Z_7} ( T_{i+3\, i+1} - T_{i+4\,i+5} -T_{i+2\,i+6} )  e_{i}^* \wedge   e_{i+5}^* \wedge e_{i+6}^* \wedge e_{i+1}^*. \nonumber
\end{align}

 Denoting  alternation by $\mathrm{alt}$, from   $(d \varphi)_{(1)} = \mathrm{alt} (\nabla \varphi)_{(1)}= \tau_0 \star \varphi =  \frac{12}{7} \sigma_1(T)\star \varphi$ and $(d \varphi)_{(4)} = \mathrm{alt} (\nabla \varphi)_{(4)} =3 \tau_1 \wedge  \varphi =  - \frac{9}{6} p(T)^\flat \wedge  \varphi$, it is obtained   
$
 \tau_0  = \tfrac{12}{7} \sigma_1(T)$,  $\tau_1 = -\tfrac12 p(T)^\flat,  
$
where it is used the musical isomorphism $x^\flat(y) = \langle x, y\rangle$, for  $x,y \in \mathrm T_m M$, $m\in M$.
Therefore, {\small
\begin{align*}
\star \tau_3  = & (d\varphi)_{(3)} = d \varphi - (d \varphi)_{(1)} - (d \varphi)_{(4)}\\
= & - 3\textstyle \sum_{i\in \mathbb Z_7} \left(T_{i+2\,i+2} + T_{i+4\,i+4} + T_{i+5\,i+5}+ T_{i+6\,i+6}- \tfrac47 \sigma_1(T) \right)  e_{i+2}^* \wedge e_{i+4}^* \wedge e_{i+5}^* \wedge e_{i+6}^*  \\ 
&  +   \tfrac32\textstyle  \sum_{i\in \mathbb Z_7} (  T_{i+1\,i+3} + T_{i+3\,i+1} +  T_{i+4\,i+5} + T_{i+5\,i+4}+T_{i+2\,i+6} + T_{i+6\,i+2} ) e_{i}^* \wedge  e_{i+1}^* \wedge e_{i+2}^* \wedge e_{i+4}^*\\
&  + \tfrac32 \textstyle    \sum_{i\in \mathbb Z_7} (-T_{i+1\,i+3} - T_{i+3\,i+1}  -T_{i+4\,i+5} -T_{i+5\,i+4} +T_{i+2\,i+6}+ T_{i+6\,i+2} )  e_{i}^* \wedge   e_{i+2}^* \wedge e_{i+3}^* \wedge e_{i+5}^*\\
&  +  \tfrac32\textstyle  \sum_{i\in \mathbb Z_7} (- T_{i+1\,i+3} - T_{i+3\,i+1} + T_{i+4\,i+5} +T_{i+5\,i+4} - T_{i+2\,i+6} - T_{i+6\,i+2} )   e_i^* \wedge   e_{i+3}^* \wedge e_{i+4}^* \wedge e_{i+6}^*\\
&  + \tfrac32 \textstyle  \sum_{i\in \mathbb Z_7} (T_{i+1\,i+3} + T_{i+3\,i+1}- T_{i+4\,i+5} -T_{i+5\,i+4} -T_{i+2\,i+6} - T_{i+6\,i+2} )  e_{i}^* \wedge   e_{i+5}^* \wedge e_{i+6}^* \wedge e_{i+1}^*.
\end{align*}  }

By one hand,  one has 
$
\tau_0^2 = \frac{12^2}{7^2} \sigma_1(T)^2 = \frac{12^2}{7} \| T_{(1)}\|^2$, and $ |\tau_1|^2 = \frac14 \|p(T)\|^2 = \frac32 \|T_{(4)}\|^2.
$

On the other hand, {\small
 \begin{align*}
| \star \tau_3 |^2 = & \, 9\textstyle  \sum_{i\in \mathbb Z_7}  \left((T_{i+1 \, i+3} + T_{i+3 \, i+1})^2 +  (T_{i+4 \, i+5} + T_{i+5 \, i+4})^2 +  (T_{i+2 \, i+6} + T_{i+6 \, i+2})^2\right)\\
 &  + 9\textstyle  \sumz \left( ( T_{i+2 \, i+2} - \tfrac17 \sigma_1(T) )^2 +  ( T_{i+4 \, i+4} - \tfrac17 \sigma_1(T) )^2\right) \\
& + 9\textstyle \sumz \left( ( T_{i+5 \, i+5} - \tfrac17 \sigma_1(T) )^2 +  ( T_{i+6 \, i+6} - \tfrac17 \sigma_1(T) )^2 \right) \\
 &+ 9 \, . \, 4\textstyle  \sumz (T_{i+1\, i+1} - \tfrac17 \sigma_1(T) )  (T_{i+3\, i+3} - \tfrac17 \sigma_1(T) )\\
& +  9 \, . \, 4\textstyle  \sumz (T_{i+4\, i+4} - \tfrac17 \sigma_1(T) )  (T_{i+5\, i+5} - \tfrac17 \sigma_1(T)) \\
&  +  9 \, . \,  4\textstyle  \sumz (T_{i+2\, i+2} - \tfrac17 \sigma_1(T) )  (T_{i+6\, i+6} - \tfrac17 \sigma_1(T)). 
\end{align*}}

Now, taking \eqref{comp3norm} into account, one has    
\begin{align*}
| \star \tau_3 |^2 = & \, 18\, . \,  \frac12 \textstyle \sum_{i\in \mathbb Z_7}  \left(  (T_{i+1 \, i+3} + T_{i+3 \, i+1})^2 +  (T_{i+4 \, i+5} + T_{i+5 \, i+4})^2 +  (T_{i+2 \, i+6} + T_{i+6 \, i+2})^2\right)\\
 &  + 9 \, . \, 4 \textstyle \sumz \left(  T_{i \, i} - \tfrac17 \sigma_1(T)  \right)^2 + 9 \, . \, 2 \sum_{i, j \in \mathbb Z_7 \, i\ne j}(T_{i\, i} - \tfrac17 \sigma_1(T) )  (T_{j\,j } - \tfrac17 \sigma_1(T) )\\= & 18 \| T_{(3)} \|^2 + 18  
 \left(\textstyle \sumz  \left(  T_{i \, i} - \tfrac17 \sigma_1(T)  \right)\right)^2 = 18 \| T_{(3)} \|^2.
 \end{align*}
 Hence $|  \tau_3 |^2 =| \star \tau_3 |^2 = 18 \| T_{(3)} \|^2$.
 \vspace{2mm}
 
 In \cite{FMC0}, it is shown an expression for $d^\star \varphi = - \star d \star \varphi$.  From such an expression one has    
 \begin{eqnarray}
d \star \varphi & = &  3\textstyle  \sum_{i \in \mathbb Z_7} \left( T_{i+1\, i+3} - T_{i+3\, i+1} -p(T)_i \right) \star(e_{i+1}^*\wedge e_{i+3}^*) \nonumber \\
&&+ 3\textstyle  \sum_{i \in \mathbb Z_7} \left( T_{i+4\, i+5} - T_{i+5\, i+4} - p(T)_i  \right) \star(e_{i+4}^*\wedge e_{i+5}^*) \label{extderstar} \\
&& + 3\textstyle  \sum_{i \in \mathbb Z_7} \left( T_{i+2\, i+6} - T_{i+6\, i+2} - p(T)_i  \right) \star(e_{i+2}^*\wedge e_{i+6}^*), \nonumber 
\end{eqnarray}
where $p(T)_i = T_{i+1\, i+3} - T_{i+3\, i+1} + T_{i+4\, i+5} - T_{i+5\, i+4} + T_{i+2\, i+6} - T_{i+6\, i+2}$. 
Since 
\begin{equation} \label{extderstar4}
(d \star \varphi)_{(4)} = 4 \tau_1 \wedge \star \varphi = - 2 p(T)^\flat \wedge  \star  \varphi,
\end{equation}
the remaining $\Lie G_2$-component  of $d \star \varphi$ is given by 
$$
(d \star \varphi)_{(2)} =  \tau_2 \wedge \varphi =  d \star \varphi - (d \star \varphi)_{(4)}.
$$ 
Therefore,
\begin{eqnarray*}
(d \star \varphi)_{(4)}  & = & - 2 \textstyle  \sum_{i \in \mathbb Z_7} p(T)_i \star (e_{i+1}^*\wedge e_{i+3}^* + e_{i+4}\wedge e_{i+5}^* + e_{i+2}^*\wedge e_{i+6}^*)  
\end{eqnarray*}
and 
 \begin{eqnarray}
(d \star \varphi)_{(2)}  & = & \textstyle \sum_{i \in \mathbb Z_7} \left( 3 (T_{i+1\, i+3} - T_{i+3\, i+1} )-  p(T)_i  \right) \star(e_{i+1}^*\wedge e_{i+3}^*)  \nonumber \\
&&+ \textstyle  \sum_{i \in \mathbb Z_7} \left( 3(T_{i+4\, i+5} - T_{i+5\, i+4} )- p(T)_i \right) \star(e_{i+4}^*\wedge e_{i+5}^*) \label{extderstar2} \\
&& + \textstyle  \sum_{i \in \mathbb Z_7} \left( 3(T_{i+2\, i+6} - T_{i+6\, i+2}) -p(T)_i  \right) \star(e_{i+2}^*\wedge e_{i+6}^*).\nonumber
\end{eqnarray}
Hence it is obtained
\begin{eqnarray*}
\tau_2  & = & 
 \frac12 \textstyle  \sum_{i \in \mathbb Z_7} \left( 6 (T_{i+1\, i+3} - T_{i+3\, i+1} )- 2 p(T)_i  \right) e_{i+1}^*\wedge e_{i+3}^*  \\
&&+  \frac12  \textstyle  \sum_{i \in \mathbb Z_7} \left( 6(T_{i+4\, i+5} - T_{i+5\, i+4} )- 2 p(T)_i  \right) e_{i+4}^*\wedge e_{i+5}^*  \\
&& +\frac12 \textstyle  \sum_{i \in \mathbb Z_7} \left(  6(T_{i+2\, i+6} - T_{i+6\, i+2}) - 2p(T)_i  \right) e_{i+2}^*\wedge e_{i+6}^*.
\end{eqnarray*}
In summary, $ \tau_2 = 6 T_{(2)}^\flat$, where $T_{(2)}^\flat(x, y) = \langle x,  T_{(2)}(y)\rangle$.
From this, it is computed
\begin{eqnarray*}
| \tau_2 |^2 & = &\textstyle \sumz (3 (T_{i+1\, i+3} - T_{i+3\, i+1} ) -  p(T)_i)^2  +   \sumz (3 (T_{i+4\, i+5} - T_{i+5\, i+4} ) -  p(T)_i)^2 \\
                   & & + \textstyle \sumz (3 (T_{i+2\, i+6} - T_{i+6\, i+2} ) -  p(T)_i)^2.
\end{eqnarray*}
This implies $|\tau_2|^2 = 36 | T_{(2)}^\flat |^2 = 18  \| T_{(2)}\|^2$. 
 \end{proof}
 Replacing the values for $\tau_i$ given in last Lemma, one found that Bryant's formula \eqref{eq:Bryant1} is agree with the one given in \eqref{formTT}. Likewise, writing the divergence equation  \eqref{eq:Bryant2} in terms of the components  $T_{(i)}$, it is checked that they are agree with \eqref{divformTT}. 
 \vspace{3mm}

 In \cite{FI}, Friedrich and Ivanov  considered types of $G$-structures which admit   $G$-connection with totally skew-symmetric torsion. They showed that for $G=\Lie G_2$ such assumption is satisfied if and only if the  $G_2$-structure is of type $\mathfrak{X}_1\oplus\mathfrak{X}_3\oplus\mathfrak{X}_4$. In fact, they proved that for such a type of $G_2$-structure, there is only one  $\Lie G_2$-connection with totally skew-symmetric torsion.  Such a torsion $\mathcal  T$ is given  (following the  conventions fixed here) by 
\begin{equation}\label{eq:torsion}
\mathcal T=- \frac{1}{6}(d\varphi,\star\varphi)\varphi+\star d\varphi+ 2 \star(p(T)^{\flat}\wedge\varphi).
\end{equation} 
In order to compute $| \mathcal T|^2$, we consider the identity 
$$
\star  \mathcal T =- \frac{1}{6}(d\varphi,\star\varphi)\star \varphi+ d\varphi+  2 p(T)^{\flat}  \wedge\varphi.
$$
From this, one has 
$$
\star  \mathcal T=- 2 \sigma_{1}(T) \star \varphi+ (d\varphi)_{(1)} + (d\varphi)_{(3)} + (d\varphi)_{(4)} + 2 p(T)^{\flat}  \wedge\varphi. 
$$
Therefore, it is finally obtained 
$$
\star \mathcal T =- \frac27  \sigma_{1}(T) \ast \varphi +\star \tau_3 +  \frac12 p(T)^{\flat}  \wedge\varphi 
$$
and 
$$
| \mathcal T|^2=  | \star  \mathcal T|^2 = \frac47  \sigma_{1}(T)^2  + | \tau_3|^2 +  \|p(T)\|^2.  
$$
This is equivalent to 
\begin{equation} \label{skewtor}
| \mathcal T|^2= 4  \| T_{(1)}\|^2  + 18 \| T_{(3)}\|^2 +  6  \| T_{(4)}\|^2 .  
\end{equation}
Friedrich and Ivanov have also deduced a divergence equation \cite{FI2} given by   
\begin{equation}\label{eq:scalarskewtorsion}
s=\frac{1}{18}(d\varphi,\star\varphi)^2+4\|p(T)\|^2-\frac{1}{2}|\mathcal T |^2+ 6 \, {\rm div}(p(T)).
\end{equation}
Note that $6 |\mathcal T |^2 = \| \mathcal T \|^2$. Now, taking $(d\varphi,\star\varphi)^2 = \frac{12^2}{7^2} \sigma_1(T)^2 =  \frac{12^2}{7} \|T_{(1)} \|^2$, $\|p(T)\|^2 = 6 \| T_{(4)}\|^2$ and equation \eqref{skewtor}   into account , it can be seen that 
the divergence equations \eqref{divformTT} and \eqref{eq:scalarskewtorsion}  are agree.


\section{The intrinsic torsion in terms of the exterior derivatives $d \varphi$ and $d \star \varphi$}

 All information about the intrinsic torsion of a $\Lie G_2$ is contained in the covariant derivative $\nabla \varphi$. A useful alternative way to find such information is by means the exterior algebra. This is possible for $\Lie G_2$-structures by studying $d \varphi$ and $d\star \varphi$ as it is described  in \cite{MMS} (see Table I in \cite{FMC0}). Because we have already displayed all the necessary tools, now we  will express the intrinsic torsion $\xi$ in terms of $d \varphi$ and of  $d \star \varphi$.  	
\vspace{3mm}

For $T_{(1)}$ and $\xi_{(1)}$, from the expression \eqref{extder} for $d \varphi$, one has
$$
( d \varphi , \star \varphi) = 3 \textstyle  \sumz (T_{i+2,i+2} + T_{i+4,i+4} + T_{i+5,i+5} + T_{i+6,i+6}) = 
12 \sigma_1(T).  
$$
Hence, taking \eqref{a1} into account, it is obtained
$$
T_{(1)} = \frac{1}{84} ( d \varphi , \star \varphi) \mathrm{Id}_{\mathrm T M}, \qquad \xi_{(1)X} Y= \frac{1}{84} ( d \varphi , \star \varphi) \, Y \times X, 
$$

For $\xi_{(4)}$, one considers the one-form $pd^\star\varphi = \star (\star d\varphi \wedge \varphi)= - \star (\star d \star \varphi \w \star \varphi)$ and we know that $p(T) = \frac16  (pd^\star\varphi)^\sharp$ by \eqref{peT}. Hence, takinq \eqref{a4}  into account, it is obtained
$$
T_{(4)} = \frac1{36} A_{(pd^\star\varphi)^\sharp}, \qquad \xi_{(4)X} Y= \frac1{36} Y \times A_{(pd^\star\varphi)^\sharp}(X) = \frac1{36} \, Y \times (X \times (pd^\star\varphi)^\sharp).  
$$
\vspace{2mm}

For $T_{(3)}$ and $\xi_{(3)}$, for all $i \in \mathbb Z_7$, one has $\xi_{(3)X}Y = Y \times T_{(3)} (X)$, where
\begin{eqnarray*}  
T_{(3)}& = & \frac16\textstyle  \sumz \left(d\varphi , e_i^*\w (e_i \lrcorner \star \varphi)   -\tfrac{4}{7} \star \varphi\right) e_i\otimes e_i^*   \\
  & & +\frac1{12} \textstyle  \sum_{i, j \in \mathbb Z_7, \, i\ne j}   \left( d\varphi , e_{i}^* \wedge (e_{j} \lrcorner \star \varphi) +  e_{j}^* \wedge (e_{i} \lrcorner \star \varphi )\right) (e_i \otimes e_j^* + e_j \otimes e_i^*).
\end{eqnarray*}

All of this can be checked by using the identity  \eqref{extder} and taking \eqref{a3dos} into account.
Note that, for all $i \in \mathbb Z_7$,   
\begin{eqnarray*}
  (p d^\star \varphi)_i^\sharp & = &  6 p(T)_i  \, = \, - \left( d\varphi , e_{i+1}^* \wedge (e_{i+3} \lrcorner \star \varphi) -  e_{i+3}^* \wedge (e_{i+1} \lrcorner \star \varphi )\right) 
\\
& = & 
- \left( d\varphi , e_{i+4}^* \wedge (e_{i+5} \lrcorner \star \varphi) -  e_{i+5}^* \wedge (e_{i+4} \lrcorner \star \varphi )\right) 
\\
 & = & - \left( d\varphi , e_{i+2}^* \wedge (e_{i+6} \lrcorner \star \varphi) -  e_{i+6}^* \wedge (e_{i+2} \lrcorner \star \varphi )\right).
\end{eqnarray*}
\vspace{3mm}
	
Finally, for $T_{(2)}$ and $\xi_{(2)}$, one has $\xi_{(2)X}Y = Y \times T_{(2)} (X)$, where
\begin{eqnarray*}  
T_{(2)} &= & \frac1{12} \textstyle  \sum_{i , j \in Z_7, \, i\ne j  } \left( d\star \varphi + \tfrac13 p d^\star \varphi \wedge \star \varphi  , \star (e_{i}^* \wedge e_{j}^*)\right) (e_i \otimes  e_j^* - e^j \otimes e_i^*).
\end{eqnarray*}
This can be checked  by using the identities  \eqref{extderstar}, \eqref{extderstar4}  and \eqref{extderstar2}, and  taking \eqref{a2uno} into account.
\vspace{3mm}

Conversely,  if we use the map $\mathsf k : \mathrm T M \otimes \mathrm T^* M \to \Lambda^4 \mathrm T^* M $, defined by 
$$
a = \textstyle  \sum_{i,j \in \mathrm Z_7}   a_{i\, j} e_i \otimes e_j^* 
\lto  \mathsf k (a) = \sum_{i,j \in \mathrm Z_7}  a_{i\, j} e_i^* \w (e_j \lrcorner \star \varphi), 
$$
we will have $(d\varphi)_{(1)} = 3 \, \mathsf k(T_{(1)})$,    $(d\varphi)_{(3)} = 3 \,\mathsf k(T_{(3)})$, and  $(d\varphi)_{(4)} = 3 \, \mathsf k(T_{(4)})$. Note that  $\mathsf k(T_{(2)}) =0$.

Now we consider  the map $\mathsf m : \mathrm T^* M \otimes \mathrm T M \to \Lambda^5 \mathrm T^* M $, defined by 
$$
a =\textstyle  \sum_{i,j \in \mathrm Z_7}   a_{i\, j} e_i \otimes e_j^*
\lto  \mathsf m (a) = \sum_{i,j \in \mathrm Z_7}  a_{i\, j} \star (e_i^* \w e_j^*). 
$$
It is obtained  $(d\star \varphi)_{(2)} = 3 \, \mathsf m(T_{(2)})$ and    $(d\star \varphi)_{(4)} = 3 \,\mathsf m(T_{(4)})$. Note that  $\mathsf m(T_{(1)}) = \mathsf m(T_{(3)}) =0$. 
\vspace{3mm}

The map $\mathsf k$    is closely related with  map (2.17) in \cite{RB} given by $\mathsf i \, : \, \mathrm T M \otimes \mathrm T^* M \lto \Lambda^3  T^* M    $, 
$$
a = \textstyle  \sum_{i,j \in \mathrm Z_7}   a_{i\, j} e_i \otimes e_j^* \lto \mathsf i(a)  = \sum_{i,j \in \mathrm Z_7}  a_{i\, i}  e_i^*\wedge (e_j \lrcorner \varphi)  
$$
In fact, it can be  checked that
$$
\mathsf k(a) + \star   \mathsf i(a^t) = \sigma_1(a) \star \varphi, 
$$ 
or, equivalently,   
$$
\star \mathsf k(a) +    \mathsf i(a^t) = \sigma_1(a)  \varphi, 
$$
for all $a \in \mathrm T M \otimes \mathrm T^* M$, where $a^t = \sum_{i,j \in \mathrm Z_7}   a_{j\, i} e_i \otimes e_j^*$.  This lead us to claim  that for two vectors $x, y  \in  \mathrm T_m M$, $m \in M$, it is satisfied the identities
$$
\star( x^* \wedge (y \lrcorner \star \varphi)) + y^* \wedge (x \lrcorner \varphi) = \langle x, y\rangle \varphi,
$$   
and
$$
 x^* \wedge (y \lrcorner \star \varphi) + \star (y^* \wedge (x \lrcorner \varphi) )= \langle x, y\rangle \star \varphi,
$$   
where, for $x= \sumz x_i e_i$, one defines $x^*= \sumz x_i e_i^*$. 
\vspace{2mm}
 
It is also interesting to have a look to map (2.18) in \cite{RB}. It is given by   $\mathsf j : \Lambda^3 \mathrm  T^*M \lto \mathrm T M \otimes \mathrm T^* M$ with
$$\mathsf j(\gamma) =\textstyle  \sum_{i,j \in \mathbb Z_7}\star ((e_i \lrcorner \varphi) \wedge (e_j \lrcorner \varphi) \wedge \gamma) \, e_i \otimes e_j^*.$$ 
  The image of $\mathsf j$ is the set of symmetric endomorphisms, i.e.    $\mathrm{Im} \, \mathsf j = \mathfrak X_1 \oplus \mathfrak X_3$. For all $a \in \mathfrak X_1 \oplus \mathfrak X_3\subset \mathrm TM \otimes \mathrm T^* M$, it is satisfied
 $$
 \mathsf j ( \mathsf i (a) ) = -4 a - 2 \sigma_1(a) \mathrm{Id}.
 $$  
In particular, one has  $\mathsf j ( \mathsf i ( \mathrm{Id} ))= - 18  \mathrm{Id}$, $\mathsf i (\mathrm{Id}) = 3 \varphi$ and $\mathsf j(\varphi) = - 6 \mathrm{Id}$. On the other hand, for all $\gamma \in \lie X_1 \oplus \lie X_3 \subset  \Lambda^3\mathrm T^* M$, it is satisfied
$$
 \mathsf i ( \mathsf j (\gamma) ) = -4 \gamma - 2 (\gamma, \varphi) \varphi.
$$
This is deduced by denoting $\gamma = \mathsf i(a)$, checking $(\gamma, \varphi) = 3 \sigma_1(a)$ and using the previous identity for $ \mathsf j ( \mathsf i (a) )$. By similar arguments, for all $a  \in \mathfrak X_1 \oplus \mathfrak X_3\subset \mathrm TM \otimes \mathrm T^* M$ and $\gamma  \in \mathfrak X_1 \oplus \mathfrak X_3\subset \Lambda^3 \mathrm T^* M$,  it can be deduced the identities
$$
\mathsf j (\star \mathsf k(a)) = - 4 a - 4 \sigma_1 (a) \mathrm{Id},  
\qquad 
 \mathsf k(\mathsf j (\gamma)) = 4 \star \gamma + \frac73 (\gamma , \varphi) \star \varphi.  
$$


\begin{thebibliography}{99}

\setlength{\baselineskip}{0.4cm}

\bibitem{Boothby}
 W. M. Boothby, \emph{An introduction to differentiable manifolds and Riemannian geometry}, Revised (Vol. 120), Gulf Professional Publishing  (2003).

\bibitem{BHLS1} G.~Bör and L.~Hern{\'a}ndez Lamoneda, {\rm Bochner formulae for orthogonal
$G$-structures on compact manifolds}, \emph{Diff. Geom. Appl. }{\bf 15} (2001), 265--286.

\bibitem{BG} R. Brown, A. Gray, Vector cross product, \emph{Comm. Math. Helv.} 42 (1967), 222--236.

\bibitem{RB} R. L. Bryant, Some remarks on $G_2$--structures, In:\emph{ Proceedings of Gokova Geometry--Topology Conference}, Gokova, 2005, 75--109.

\bibitem{CleytonSwann:torsion}
R.~Cleyton and A.~F.~Swann, {E}instein metrics via intrinsic or
parallel torsion, \emph{Math. Z.} 247 no.~3 (2004), 513--528.

\bibitem{FG} M. Fern\'{a}ndez, A. Gray, Riemannian manifolds with structure group 
$\Lie G_2$, \emph{ Ann. Mat. Pura Appl.} (4) 132 (1982), 19--45. 

\bibitem{FI} T. Friedrich, S. Ivanov, Parallel spinors and connections with skew-symmetric torsion in string theory,\emph{Asian J. Math.} 6 (2002), no. 2, 303--335.

\bibitem{FI2} T. Friedrich, S. Ivanov, Killing spinor equations in dimension 7 and geometry of integrable $\Lie G_2$-manifolds, \emph{J. Geom. Phys.} 48 (2003), no. 1, 1--11. 

\bibitem{Gray:vector}
 A. Gray, Vector cross products on manifolds, \emph{Trans. Am. Math.
Soc.} 141 (1969), 465--504.

\bibitem{KobNom}
S. Kobayashi, K.  Nomizu, \emph{Foundations of differential geometry} (Vol. 1, No. 2). New York, London   (1963).

\bibitem{FMC0} F. Martín Cabrera, $\Lie{Spin}(7)$-structures on principal bundles over Riemannian manifolds with
$\Lie G_2$-structure, \emph{ Rend. Circ. Mat. Palermo} (2) 44 (1995),  249--272.


\bibitem{FMC00} 
F.  Martín Cabrera, On Riemannian manifolds with $\Lie G_2$-structure, \emph{ Boll. Unione Mat. It.} (7)
9-A  (1996), 99--112


\bibitem{FMC} F. Martín Cabrera, $\Lie{SU}(3)$--structures on hypersurfaces of manifolds with $G_2$--structure, \emph{Monatsh. Math.} 148 (2006), 29--50.

\bibitem{MMS} 
F.  Martín Cabrera, M.  D. Monar, Classification of   $\Lie G_2$-structures, \emph{J. London Math. Soc.} 53
  (1996), 407--416.


\bibitem{KN2} K. Niedzia\l omski, An integral formula for  $\Lie{G}_2$--structures, \emph{J. Geom. Phys.}  (to appear).

\bibitem{ON}
B. O'Neill,   \emph{Semi-Riemannian Geometry with Applications to Relativity}: Volume 103 (Pure and Applied Mathematics) Academic Press, New York (1983).
\end{thebibliography}
\end{document}